\documentclass[reqno]{amsart}
\newcommand \datum {August 3, 2020}
\usepackage{amssymb,latexsym}
\usepackage{amsmath}
\usepackage{graphicx}
\usepackage{amscd}
\usepackage{color}
\usepackage{enumerate}
\newenvironment{enumeratei}{\begin{enumerate}[\upshape (i)]}{\end{enumerate}}
\numberwithin{equation}{section}
\theoremstyle{plain}
 \newtheorem{theorem}{Theorem}[section]
 \newtheorem{lemma}[theorem]{Lemma}
 
 \newtheorem{observation}[theorem]{Observation}

\theoremstyle{definition}
 
 \newtheorem{remark}[theorem]{Remark}

\theoremstyle{remark}
 
 \newtheorem{case}{Case}

\newcommand \notdiv {\mathrel{\not{\kern -0.27pt|}} }
\newcommand \Nplus {\mathbb N^+}
\newcommand \Nnul {\mathbb N_0}

\newcommand \ffree[1] {\textup{FL}(3)}
\newcommand \FL[1] {\textup{FL}(#1)}
\newcommand \fmth {\textup{FM}(3)}

\newcommand \Ats[1] {\textup {At}({#1})}

\newcommand \con [1] {\textup{con}(#1)}

\newcommand \freelat [1] {\textup{FL}\bigl(#1\bigr)}
\newcommand \vxi {\vec \xi\,}

\newcommand \blokk [2] {#1/#2}
\newcommand \lvar [1] {\mathcal #1}

\newcommand \Ngen [1] {\lvar L_{#1\textup{-gen}}}

\newcommand \ideal{\mathord{\downarrow}}
\newcommand \filter{\mathord{\uparrow}}
\newcommand \lideal[1]{\mathord{\downarrow}_{\kern-1pt#1}}
\newcommand \lfilter[1]{\mathord{\uparrow}_{\kern-1pt#1}}

\newcommand \uu {\widetilde u}
\newcommand \vv {\widetilde v}
\newcommand \ww {\widetilde w}

\newcommand \tbf [1] {\textbf{#1}} 
 
\newcommand \set[1] {\{#1\}}

\renewcommand \phi {\varphi}

\newcommand \red [1] {{\color{red}#1\color{black}}}

\newcommand \nothing [1] {}
\newcommand \magenta [1] {{\color{magenta}#1\color{black}}}

\begin{document}
\title[On atoms in three-generated lattices] 
{On the number of atoms in three-generated lattices}

\author[G.\ Cz\'edli]{G\'abor Cz\'edli}
\email{czedli@math.u-szeged.hu}
\urladdr{http://www.math.u-szeged.hu/\textasciitilde{}czedli/}
\address{University of Szeged, Bolyai Institute, 
Szeged, Aradi v\'ertan\'uk tere 1, HUNGARY 6720}

\thanks{This research was supported by the Hungarian Research, Development and Innovation Office under
grant number KH 126581.}

\dedicatory{Dedicated to the memory of Professor Gyula Pap $(\kern-1pt$1954--2019$\kern1pt)$
}

\date{\hfill {\tiny{\magenta{(\tbf{Always} check the author's website for possible updates!) }}}\  \red{\datum}}

\subjclass[2000] {06B99}
\keywords{Three-generated lattice,  number of atoms, coatom, atomless lattice, herringbone lattice, $n$-distributive lattice, 2-distributive lattice, non-finitely presented lattice, convex geometry, meet-distributive lattice, semidistributive lattice, semimodular lattice}

\begin{abstract} 
As the main achievement of the paper, we construct a three-generated, 2-distributive,  atomless lattice that is not finitely presented. Also, the paper contains the following three observations. First, every coatomless three-generated lattice has at least one atom. Second, we give some sufficient conditions implying that a three-generated lattice has at most three atoms. Third, we present a three-generated meet-distributive lattice with four atoms.
\end{abstract}

\maketitle
\section{Result and introduction}\label{secgoal}
Our main goal is to prove the following theorem; the corresponding  (widely known) definitions are postponed to Section~\ref{sectdefintro}.

\begin{theorem}\label{thmmain} There exists a three-generated 
lattice $L$ such that  
\begin{enumeratei}
\item\label{thmmaina} $L$ has no atom,
\item\label{thmmainb} $L$ is $2$-distributive, and
\item\label{thmmainc} $L$ is not finitely presented.
\end{enumeratei}
\end{theorem}

In the proof of this theorem, a  three-generated lattice $L$ satisfying \eqref{thmmaina}--\eqref{thmmainc} will concretely be constructed. 
Also, we are going to verify two observations. 

\begin{observation}\label{observtr} Let $L$ be a lattice generated by a three-element subset $\set{a_0,a_1,a_2}$. 
\begin{enumeratei}
\item\label{observtra}
If this three-generated lattice has no coatom, then it has at least one atom.
\item\label{observtrb} With the notation
$k:=|\set{(i,j)\in  
\set{(0,1), (0,2),(1,2)}: a_i\wedge a_j\neq a_0\wedge a_1\wedge a_2}|$, if $k\in\set{2,3}$, then $L$ has exactly $k$ atoms.
\item\label{observtrc} If $L$ is modular, then $L$ has at most three atoms and it has at least one.
\end{enumeratei}
\end{observation}

Note that \eqref{observtra} above is a particular case of Freese~\cite[equation (10)]{freese89}; see also  Freese and Nation~\cite[Theorem 2-7.2]{fn}. Postponing the definitions to Section~\ref{sectdefintro} again, we formulate the second observation as follows. 

\begin{observation}\label{observnD}\ 
There exists a twelve-element three-generated meet-distributive lattice with exactly four atoms. 
\end{observation}

\subsection*{Outline}
Section~\ref{sectdefintro} gives some basic definitions and recalls some facts motivating the present paper. In Section~\ref{sectproofs}, we prove Theorem~\ref{thmmain}; Remark~\ref{remhrrBnNl} on the herringbone lattice and  Lemma~\ref{lemmadzghrwtNs} of this section can be of separate interest. Finally, Observations~\ref{observtr} and \ref{observnD} are proved in Section~\ref{sectobsproof}.

\section{Basic concepts, motivation, and some related results}\label{sectdefintro}
A sublattice $S$ of a lattice $L$ is  \emph{proper} if $S\neq L$. A  lattice $L$ is 
\emph{three-generated} if it has a three-element subset $\set{a_1,a_2,a_3}$  such that  $\set{a_1,a_2,a_3}\subseteq S$ holds for no proper sublattice $S$ of $L$. For an element $a$ in a lattice $L$, the  \emph{principal ideal} $\set{x\in L: x\leq a}$ and the \emph{principal filter} $\set{x\in L: a\leq x}$ will be denoted by $\ideal a$ and $\filter a$, respectively. Note that $a$ is an \emph{atom} (of $L$) iff $|\ideal a|=2$, and it is a \emph{coatom} iff $|\filter a|=2$. Let $\Ats L$  stand for the set of atoms of $L$.

The class $\Ngen 3$ of three-generated lattices is quite large and involved. For example, this class contains $2^{\aleph_0}$ many non-isomorphic members and 
every lattice $L$ of size at most $\aleph_0$ is a sublattice of a lattice in $\Ngen 3$; see Crawley and Dean~\cite[Theorem 7]{crawleydean}. As a related result, it was proved in 
Cz\'edli~\cite[Corollary 1.3]{czgembin3} that every finite lattice can be embedded in a \emph{finite} member of $\Ngen 3$. However, for each three-generated lattice $L$ that the author has ever seen in the literature, including  Cz\'edli~\cite{czgembin3}, Davey and Rival~\cite{daveyrival}, Freese, Je\v zek, and Nation~\cite{fjn}, Gr\"atzer~\cite{gratzer}, and Poguntke~\cite{poguntke}, we have that  $|\Ats L|\in\set{1,2,3}$. Now, from Theorem~\ref{thmmain}\eqref{thmmaina} and Observation~\ref{observnD}, we learn that $|\Ats L|=0$ and $|\Ats L|=4$ are also possible. 

We know more about the atoms of \emph{four-generated} lattices than those of the three-generated ones. Four-generated lattices can have very many atoms; without seeking completeness, we only list some relevant results and facts below. Finite equivalence lattices have many atoms and these lattices are four-generated by  Strietz~\cite{strietz}; see also Z\'adori~\cite{zadori4g} for a nice proof. The lattices of quasiorders over finite base sets $A$ with $|A|>10$ are also four-generated and have many atoms by Cz\'edli~\cite{czgquo4gen} and Cz\'edli and Kulin~\cite{czgkulin}, and there are also analogous results 
over infinite base sets in  \cite{czgquo4gen}, \cite{czgkulin}, and  Cz\'edli~\cite{czedlifourgen,czedlioneonetwo}. There are modular examples as well since the subspace lattice $L(n,F)$ of an $n$-dimensional vector space over a prime field $F$ is four-generated for every integer $n\geq 3$ by Gelfand and Ponomarev~\cite{gelfandponom}; see also Z\'adori~\cite{zadori5gen} for an analogous result and an overview. Two particular cases are worth mentioning about these four-generated lattices:  if $F=\mathbb Q$, the field of rational numbers, then $L(n,F)$ has $\aleph_0$-many atoms while  if $n=3$, then  $L(n,F)$ is generated by four of its \emph{atoms} by Herrmann and Huhn~\cite{herrmannhuhn}. As one would expect, there are four-generated lattices without atoms. In view of Observation~\ref{observtr}\eqref{observtra}, 
the following result proved by
Freese~\cite[Section 6]{freese89}, see also
Freese and Nation~\cite[Theorem 2-7.5]{fn},  is worth mentioning: there exists a four-generated lattice that has no two-element interval at all; clearly, this lattice is atomless and coatomless.

The theory of meet-distributive lattices  goes back to
Dilworth~\cite{dilworth}; see also 
Adaricheva, Gorbunov, and  Tumanov~\cite{adarichevaatal},
Edelman~\cite{edelman}, Edelman and Jamison~\cite{edelmanjamison}, and other papers referenced by \cite{czgcoord}. These lattices are the lattice theoretical counterparts of abstract convex geometries. By definition, a finite lattice $L$ is \emph{meet-distributive} if for each 
$x\in L$, there is a \emph{unique} minimal set $Y$ of join-irreducible elements such that $x=\bigvee\set{y: y\in Y}$. Many other definitions are listed in Monjardet~\cite{monjardet}. A survey and  some more definitions are given Cz\'edli~\cite{czgcoord};  see  Lemma 7.4 and  the dual of Proposition 2.1 there. Yet another description of these lattices is provided by the dual of Proposition 6.1 of Adaricheva and Cz\'edli~\cite{adariczg}.

The concept of $n$-distributive lattices was introduced by Huhn~\cite{huhnschwacha,huhnschwachb}. Due to its links to von Neumann's coordinatization theory, see Herrmann and Huhn~\cite{herrmannhuhn}, to convex geometry, see Huhn~\cite{huhnconv} and Libkin~\cite{libkin}, and to various questions in lattice theory, see, for example, Huhn~\cite{huhnautom}, this concept soon became important in lattice theory. While 1-distributive lattices are the usual distributive ones and well studied, 2-distributive ones are of special importance; see, for example, J\'onsson and Nation~\cite{jonssonnation}. Here we only define 2-distributivity; a lattice $L$ is \emph{$2$-distributive} if 
\begin{equation}
x\wedge (y_0\vee y_1\vee y_2)\leq \bigl( x\wedge (y_0\vee y_1)\bigr) \vee \bigl( x\wedge (y_0\vee y_2)\bigr) \vee \bigl( x\wedge (y_1\vee y_2)\bigr)
\end{equation}
holds for all $x,y_0,y_1,y_2\in L$.

A lattice $L$ is \emph{finitely presented} if there is a positive integer $n$ and there are finitely many $n$-ary lattice terms $f_1',f_1'',\dots, f_t',f_t''$ such that $L$ is isomorphic to 
\begin{equation}
\freelat{x_1,x_2,\dots,x_n}/\Theta, \text{ where } \Theta:=\bigvee_{i=1}^t\con{f_i'(\vec x),f_i''(\vec x)},
\label{eqfLPrshcLt}
\end{equation}
$\freelat{x_1,x_2,\dots,x_n}=:\FL n$ is the lattice freely generated by the $n$-element set $\set{x_1,x_2,\dots,x_n}$ in the variety of all lattices, $\vec x$ abbreviates $(x_1,x_2,\dots,x_n)$,  $\con{u,v}$ denotes the least congruence collapsing $u$ and $v$, and the join is taken in the lattice of all congruences of $\FL n$. In other words, quotient lattices of finitely generated free lattices modulo finitely generated congruences are said to be finitely presented.  Our standard notation for the lattice in  \eqref{eqfLPrshcLt} is 
\begin{equation}
\freelat{u_1,\dots,u_n: \,f_1'(\vec u)=f_1''(\vec u),\dots, f_t'(\vec u)=f_t''(\vec u)};
\label{eqdNJssKlPzF}
\end{equation}
here $u_i$ is $x_i/\Theta$, $\vec u=(u_1,\dots,u_n)$,  and the lattice is generated by $\set{u_1,\dots, u_n}$. Note that every finite lattice is finitely presented. 
Usually, being finitely presented is considered a positive property. In case of finitely generated \emph{infinite} lattices, it is the lack of  this property that we consider positive in this paper, because we feel that taking an infinite join in \eqref{eqfLPrshcLt} allows us to encode more information in $\Theta$ 
and to obtain a more structured and less complicated $\freelat{x_1,\dots,x_n}/\Theta$ in many cases.
%

\section{Proving the main result}\label{sectproofs}
In this section, we prove our main result, Theorem~\ref{thmmain}. Also, this section contains Remark~\ref{remhrrBnNl} and Lemma~\ref{lemmadzghrwtNs}, which can be of separate interest.

\begin{proof}[Proof of Theorem~\ref{thmmain}] 
Let $H$ be the \emph{herringbone lattice}  in the middle of Figure~\ref{figrc}. This lattice has
played important roles in several papers including
Bauer and Poguntke~\cite{bauerpoguntke}, Poguntke~\cite[Figure 10]{poguntke}, Poguntke and Sands~\cite{poguntkesands}, Rival, Ruckelshausen, and Sands~\cite{rivalatal}, Rolf~\cite{rolf}, and Wille~\cite{wille}.
We know from these papers that 
\begin{equation}
\text{$H$ is a three-generated  lattice and it  is generated by $\set{a,b,c}$,}
\label{eqtxtHthRgltNs} 
\end{equation}
that is, by the black-filled elements in the figure;  for later reference, we are going to prove this fact below in few lines. Let $\vxi:=(x,y,z)$. For each $u\in H$, we are going to define a ternary term $g_u(\vxi)$ by induction with the  purpose that 
\begin{equation}
g_u(a,b,c)=u\, \text{ should hold in }H\text{ for all }u\in H.
\label{eqhsHdlKrSnFr}
\end{equation}
So, with $i\in \Nplus:=\set{1,2,3,\dots}$, we let
\begin{align}
&g_{a_0}(\vxi)=g_{a}(\vxi):=x,  &&g_b(\vxi):=y, 
\label{eqhzGfHKmma}
\\
&g_{c_1}(\vxi)=g_c(\vxi):=z, &&g_{q_0}(\vxi):= g_{a}(\vxi)\vee g_b(\vxi),
\label{eqhzGfHKmmb}\\ 
&g_0(\vxi):=g_{a}(\vxi)\wedge g_b(\vxi)\wedge g_c(\vxi),&& g_{q_1}(\vxi):=g_b(\vxi)\vee g_c(\vxi),
\label{eqhzGfHKmmc}\\
&g_{a_{2i}}(\vxi):=g_{a}(\vxi)\wedge g_{q_{2i-1}}(\vxi),&&
g_{q_{2i}}(\vxi):=g_{a_{2i}}(\vxi)\vee g_b(\vxi),
\label{eqhzGfHKmmd}\\
&g_{c_{2i+1}}(\vxi):=g_{c}(\vxi)\wedge g_{q_{2i}}(\vxi),&&
g_{q_{2i+1}}(\vxi):=g_{c_{2i+1}}(\vxi)\vee g_b(\vxi).
\label{eqhzGfHKmme}
\end{align}
It is clear by Figure~\ref{figrc} that \eqref{eqhsHdlKrSnFr} holds, whereby $H$ is indeed generated by $\set{a,b,c}$.

In the direct square $H\times H$, after letting
$\uu=(a,b)$, $\vv=(b,a)$, and $\ww=(c,c)$, 
\begin{equation}
\text{we define $L$ as the sublattice generated by $\set{\uu,\vv,\ww}$.}
\label{eqtxtLdef}
\end{equation} 
Clearly, $L$ is a three-generated lattice by its definition; we are going to prove that it has all the required properties.

\begin{figure}[ht]
\centerline
{\includegraphics[scale=0.9]{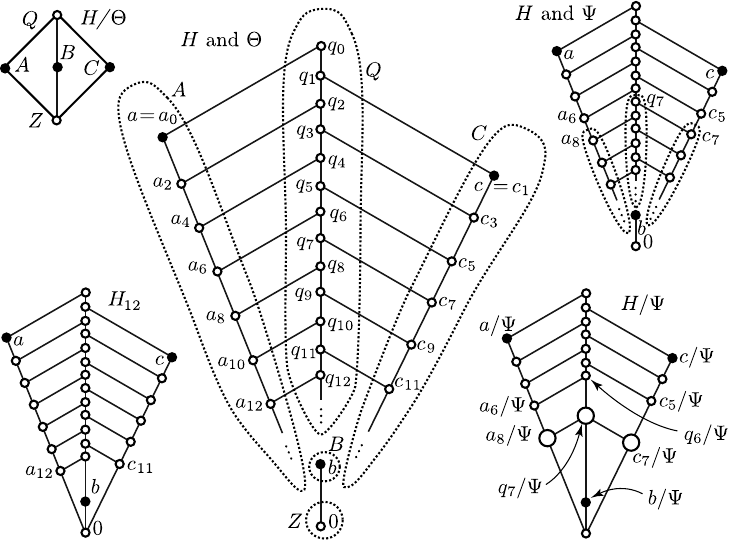}}
\caption{The herringbone lattice $H$ and $\Theta$ in the middle;  $H/\Theta$ and  $H_{12}$ on the left;  $H$, $\Psi$, and $H/\Psi$ on the right}
\label{figrc}
\end{figure}

In order to show that $L$ has no atom, first we show that 
\begin{equation}\left.
\parbox{8.2cm}{if $(x_1,x_2)\in L$ such that either $x_1\in H\setminus\set 0$ and $x_1$ is not an atom in $H$, or $x_2\in H\setminus\set 0$ and $x_2$ is not an atom in $H$,  then $(x_1,x_2)$ is not an atom in $L$.}\,\,\,
\right\}
\label{pbxSchRszLl}
\end{equation}
By symmetry, it suffices to deal with the case when the premise of \eqref{pbxSchRszLl} stipulates a condition on $x_1$.  So we assume that  $x_1\in H\setminus\set 0$ is not an atom in $H$, and pick an element $d\in H$ such that $0<d<x_1$. Then there is a ternary lattice term $t$ such that $d=t(a,b,c)$. Since
\begin{align*}
(x_1,x_2)\wedge t(\uu,\vv,\ww)&=(x_1,x_2)\wedge (t(a,b,c),t(b,a,c)) \cr
&=(x_1\wedge t(a,b,c),x_2\wedge t(b,a,c)) = (d,x_2\wedge t(b,a,c))
\end{align*}
is strictly less than $(x_1,x_2)$ and it is not $(0,0)$,
we conclude that $(x_1,x_2)$ is not an atom in $L$. This proves \eqref{pbxSchRszLl}. 

Next, we consider the equivalence relation $\Theta$ on $H$ whose blocks are $A=\set{a=a_0,a_2,a_4,a_6,\dots}$, $B=\set{b}$,
$C=\set{c=c_1,c_3,c_5,\dots}$, $Q=\set{q_0,q_1,q_2,q_3,\dots}$ and $Z=\set0$, as it is indicated by dotted ovals in the middle of Figure~\ref{figrc}. 
Clearly, $\Theta$ is a congruence and the quotient lattice $H/\Theta$ is $M_3=\set{Z,A,B,C,Q}$; see at the top left of Figure~\ref{figrc}. Let $\phi\colon H\to M_3$ denote the natural projection, that is, $\phi(x)=A$ iff $x\in A$, $\phi(x)=B$ iff $x\in B$, \dots, $\phi(x)=Q$ iff $x\in Q$. 
We claim that for every ternary lattice term $t$,
\begin{equation}\left.
\parbox{6.6cm}{$t(a,b,c)= b$ implies that $t(b,a,c)\in A$, and  $t(b,a,c)= b$ implies that $t(a,b,c)\in A$.}\,\,
\right\}
\label{eqtxtchScRWl}
\end{equation}
In order to show this, let $\psi\colon M_3\to M_3$ denote the unique automorphism of $M_3$ such that $\psi(A)=B$, $\psi(B)=A$, and $\psi(C)=C$. Assume that $t(a,b,c)=b$. Applying $\phi$ to this equality, we obtain that $t(A,B,C)=t(\phi(a),\phi(b),\phi(c))=\phi(t(a,b,c))=\phi(b)=B$. Hence, using that both $\phi$ and $\psi$ commute with $t$ and that $u\in\phi(u)$ for every $u\in H$ (by the definition of $\phi$), we can compute as follows. \begin{align*}
t(b,a,c)\in \phi(t(b,a,c))= t(\phi(b),\phi(a),\phi(c))=
t(B,A,C)\cr
=t(\psi(A),\psi(B),\psi(C))=\psi(t(A,B,C))=\psi(B)=A.
\end{align*}
This proves the first half of \eqref{eqtxtchScRWl}. The second half follows similarly. Alternatively, the second half follows immediately by applying the first half to the auxiliary ternary term $\widehat t(x,y,z):=t(y,x,z)$. Therefore,  \eqref{eqtxtchScRWl} holds.

Next, for the sake of contradiction, suppose that $L$ has an atom $(x_1,x_2)$. For an appropriate ternary lattice term $t$, we have that $(x_1,x_2)=t(\uu,\vv,\ww)$, that is,
\begin{equation}
x_1=t(a,b,c)\quad\text{ and }\quad x_2=t(b,a,c).
\label{eqshrtRmmMckrM}
\end{equation}
Since $(x_1,x_2)$ is an atom, at least one of $x_1$ and $x_2$ is nonzero. First, assume that $x_2\neq 0$. Then  \eqref{pbxSchRszLl} yields that $x_2$ is an atom in $H$. Since $b$ is the only atom of $H$, we obtain that $t(b,a,c)=b$. 
Hence, the second half of \eqref{eqtxtchScRWl} implies that $x_1=t(a,b,c)\in A$. Using that $0\notin A$,  \eqref{pbxSchRszLl} gives that $x_1$ is an atom in $H$. This is a contradiction since $A$, being an infinite descending chain, contains no atom of $H$. Second, the assumption $x_1\neq 0$ leads to the same contradiction similarly; the only difference is that now the first half of  \eqref{eqtxtchScRWl} is needed. We have shown that part \eqref{thmmaina} of Theorem~\ref{thmmain} holds, that is,  $L$ has no atom.

We say that a lattice $K=(K;\vee,\wedge)$ is of \emph{breadth at most 2} if for every nonempty  subset $X$ of $K$, there are $x_1,x_2\in X$ such that $\bigvee X=x_1\vee x_2$. It belongs to the folklore that 
\begin{equation}
\text{every planar lattice is of breadth at most 2.}
\label{eqtxtplLtBrdTht}
\end{equation}
Since we had no direct reference to this fact while writing Cz\'edli, Powers, and White~\cite[see (1.6) in it]{czgpowwhite}, we presented a proof of \eqref{eqtxtplLtBrdTht} there. A shorter proof can be obtained by combining  Lemma 3.12(B,C) (cited from Kelly and Rival~\cite{kellyrival}) and  Proposition 3.13(A) of Cz\'edli~\cite{czgfgcircles}.
Note at this point that a planar lattice is \emph{finite} by definition.
We claim that the herringbone lattice $H$ is of breadth at most 2. 
Clearly, this property only depends on the join-semilattice reduct $(H;\vee)$ of $H=(H;\vee,\wedge)$. 
For a positive  integer $m$, let 
\begin{equation}
\begin{aligned}
H_m:=&\set{a_i: 0\leq i\leq m\text{ and } 2\mid i} \cup \set{q_i: 0\leq i\leq m}\cr
&\cup \set{c_i: 0\leq i\leq m\text{ and } 2\notdiv i} \cup \set{0,b}.
\end{aligned}
\label{eqHndfLstkW}
\end{equation}
For $m=12$, $H_m$ is given on the right of Figure~\ref{figrc}. Clearly, $H_m$ is a join subsemilattice of $(H;\vee)$ and 
$H$ is the (directed) union of these subsemilattices. Hence, to show that $(H;\vee)$ is of breadth at most 2, it suffices to show that so are the $(H_m;\vee)$ for all integers $m\geq 1$. But this holds by 
\eqref{eqtxtplLtBrdTht}, and we conclude that the lattice $H=(H;\vee,\wedge)$ is of breadth at most 2. This property of $H$ trivially implies that 
\begin{equation}
\text{$H$ is 2-distributive.}
\label{eqtxthTsdsrbv}
\end{equation}
Since lattice identities are preserved by forming direct squares and taking sublattices, we conclude that $L$ is 2-distributive, proving  part \eqref{thmmainb} of Theorem~\ref{thmmain}.

Next we recall a part of Corollary 3.2 from Mair and  Ru\v skuc~\cite{mairruskuc}; we omit the middle sentence from this corollary and we give a concise formulation. If an algebra $C$ is a subdirect product of algebras $A$ and $B$, then $C$ is a subalgebra of $A\times B$ and the restrictions of the projections $A\times B\to A$, defined by $(x,y)\mapsto x$, and $A\times B\to B$ to $C$, defined by $(x,y)\mapsto y$,  will be denoted by $\pi_A$ and $\pi_B$, respectively.
Note that  \eqref{eqpbxRuskucMair} below tailors a condition on $\pi_B(\ker(\pi_A)\vee \ker(\pi_B))$, but we will not have to understand what this congruence means when \eqref{eqpbxRuskucMair} is applied to our situation. 
\begin{equation}\left.
\parbox{8.3cm}{Assume that $C$ is a subdirect product of $A$ and $B$ in a congruence modular variety, $C$ is finitely presented, and 
the congruence $\pi_B(\ker(\pi_A)\vee \ker(\pi_B))$ of $B$ is finitely generated. Then $A$ is finitely presented.}\,\,\,
\right\}
\label{eqpbxRuskucMair}
\end{equation}
As it is clear from \eqref{eqpbxRuskucMair} and from the rest of  Mair and  Ru\v skuc~\cite{mairruskuc}, the connection between the finite presentability of subdirect products and that of their subdirect factors is more complicated than we could, possibly, expect. This is so even if direct products rather than subdirect ones are considered; see Mair and  Ru\v skuc~\cite{oldmairruskuc}.

In order to make \eqref{eqpbxRuskucMair} applicable for our purpose, we are going to prove the following two statements:
\begin{align}
&\text{The herringbone lattice $H$ is not finitely presented, }
\label{eqtxthbnFnpRsd}\\
&\text{and every congruence of $H$ is finitely generated.}
\label{eqtxdhRrngfGcnRcs}
\end{align}

For the sake of contradiction, suppose that \eqref{eqtxthbnFnpRsd} fails. This means that $H$ is finitely presented, whence it is of the form
\begin{equation}
H=\freelat{u_1,\dots,u_n: \,f_1'(\vec u)=f_1''(\vec u),\dots, f_t'(\vec u)=f_t''(\vec u)},
\label{eqdjMbnBHglrSg}
\end{equation}
where $n,t\in \Nplus$,  $\vec u=(u_1,\dots, u_n)$, and the $f_i'$ and $f_i''$ are $n$-ary lattice terms; see \eqref{eqdNJssKlPzF} for more details about this notation. 
Since $H$ is  generated by $\set{u_1,\dots, u_n}$, there are $n$-ary lattice terms $h_a$, $h_b$, and $h_c$ such that 
\begin{equation}
h_a(\vec u)=a,\quad h_b(\vec u)=b,\quad\text{and}\quad h_c(\vec u)=c\quad\text{hold in }\,H.
\label{eqHvsVlgcjSWrnpMk}
\end{equation}
Next, we are going to use $H_m$ defined in \eqref{eqHndfLstkW} for each $m\in\Nplus$. Note that  $H_m$ is join-subsemilattice but not a sublattice of $H$; however, $H_m$ happens to be a lattice with respect to the ordering inherited from $H$. It is straightforward to see that
\begin{equation}
\set{p\wedge q: p,q\in H_m} \cup \set{p\vee q: p,q\in H_m}
\subseteq H_{m+1}
\label{eqxhWrvzSsWnknSkHp}
\end{equation}
holds for every $m\in\Nplus$.
Since  only finitely many elements $u_i$ and finitely many terms $f_j'$,  $f_j''$, $h_a$, $h_b$, and $h_c$ occur in \eqref{eqdjMbnBHglrSg}--\eqref{eqHvsVlgcjSWrnpMk}, and these terms contain only finitely many join and meet operation signs, it will soon follow from \eqref{eqxhWrvzSsWnknSkHp}
that we can choose an integer $m\in\Nplus$ such that 
\begin{equation}\left.
\parbox{9.3cm}{$|H_m|\geq 10$, $\set{a,b,c,u_1,\dots, u_n}\subseteq H_m$ and, for every $j\in\set{1,\dots, t}$, the equality $f_j'(\vec u)= f_j''(\vec u)$ holds in the lattice $H_m$ as well as the 
equalities $h_a(\vec u)=a$, $h_b(\vec u)=b$, and $h_c(\vec u)=c$.}\,\,\,
\right\}
\label{eqpbxshZggMbkm}
\end{equation}
Indeed, we can pick an $m_0$ such that $\set{a,b,c,u_1,\dots, u_n}\subseteq H_{m_0}$. Let, say $f_1'(\vec x)=\bigl((x_1\wedge x_2)\vee (x_3\wedge x_4)\bigr)\wedge x_5\wedge x_6$. (This is an example carrying the general idea satisfactorily.)
In the next few lines while we are proving \eqref{eqpbxshZggMbkm}, $\vee$ and $\wedge$ are understood in $H$. 
 With $m_1:=m_0+1$, it follows from \eqref{eqxhWrvzSsWnknSkHp} that $H_{m_1}$ contains $u_1\wedge u_2$ and $u_3\wedge u_4$. Since it is a join-subsemilattice of $H$, $H_{m_1}$ also contains 
$(u_1\wedge u_2)\vee(u_3\wedge u_4)$. In the next step,
with $m_2:=m_1+1$, we conclude by  \eqref{eqxhWrvzSsWnknSkHp} that $H_{m_2}$ contains
$\bigl((u_1\wedge u_2)\vee (u_3\vee u_4)\bigr)\wedge u_5$. 
In the next step, with $m_3:=m_2+1$, we obtain similarly that  $H_{m_2}$ contains
$\bigl((u_1\wedge u_2)\vee (u_3\vee u_4)\bigr)\wedge u_5\wedge u_6=f_1'(\vec u)$. We can proceed similarly by increasing the subscript of $H$ one by one, and finally we obtain a  subscript $m$ large enough such that 
all terms occurring in \eqref{eqpbxshZggMbkm} and their subterms behave in the same way in $H_m$ as in $H$. 
If $|H_m|\geq 10$ fails, then we can increase $m$. This proves \eqref{eqpbxshZggMbkm}.

\begin{figure}[ht]
\centerline
{\includegraphics[scale=1.0]{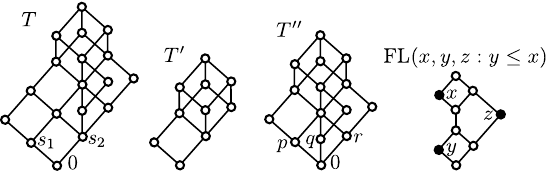}}
\caption{$T$, $T'$, $T''$, and $\freelat{x,y,z: y\leq x}$}
\label{figrb}
\end{figure}

Observe that as the inductive definition of the terms occurring in \eqref{eqhzGfHKmma}--\eqref{eqhzGfHKmme} proceeds, the subscripts of subscripts are increased one by one. Hence, $\set{a,b,c}$ generates the lattice $H_m$. Combining this fact with the last three equalities of \eqref{eqpbxshZggMbkm}, we obtain that $\set{u_1,\dots, u_n}$ also generates $H_m$. 
By  \eqref{eqpbxshZggMbkm}, $\set{u_1,\dots, u_n}$  is such a generating set of $H_m$ that satisfies the ``defining equalities'' $f_i'(\vec u)=f_i''(\vec u)$ occurring in \eqref{eqdNJssKlPzF}. Therefore, by von Dyck's theorem, $H_m$ is a homomorphic image of $H$. Hence, we can
\begin{equation}
\text{pick a congruence $\Psi$ of $H$ such that $H/\Psi\cong H_m$.}
\label{eqtxtpCktvGkltsn}
\end{equation} 
It is well known that the blocks of $\Psi$ (as well as those of any congruence) are convex sublattices; see, for example,
Gr\"atzer~\cite[Lemma 3.10]{gratzer}. Since $H_m$ is finite but $H$ is not, $\Psi$ has at least one non-singleton block. Using that each non-singleton interval of $H$ has a prime interval $[p,r]$ (that is, an edge $p\prec r$ in the diagram), it follows that $\Psi$ collapses an edge $[p,r]$.
There are three cases depending on the orientation of the edge $[p,r]$ in the middle of Figure~\ref{figrc}.

\begin{case}\label{caseone} We assume that $[p,r]$ is parallel to $[a_0,q_0]$.
Clearly, $[p,r]$ is up-perspective to $[a_0,q_0]$, that is, 
$a_0\wedge r=p$ and  $a_0\vee r=q_0$.  Since perspective intervals generate the same congruence and so they belong to the same congruences, $(a,q_0)=(a_0,q_0)\in \Psi$. This containment and
$b\leq q_0$ give that $b/\Psi\leq a/\Psi$.  Clearly, 
$H/\Psi$ is generated by $\set{a/\Psi,b/\Psi,c/\Psi}$ since $\set{a,b,c}$ generates $H$. Thus, $H/\Psi$ is a homomorphic image of the finitely presented lattice $\freelat{x,y,z: y\leq x}$, which consists of nine elements by, say, Gr\"atzer~\cite[Figure 6]{gratzer}; see also on the right of Figure~\ref{figrb}. This contradicts $|H/\Psi|=|H_m|\geq 10$ and excludes this case.
\end{case}

\begin{case}\label{casetwo} We assume that $[p,r]$ is parallel to $[c_1,q_1]$. Then, analogously to the previous case,
$b/\Psi\leq c/\Psi$ and so $10\geq |H_m|=|H/\Psi|\leq |\freelat{x,y,z: y\leq z}|=9$ is a contradiction excluding this case.
\end{case}

For later reference, let us summarize that
\begin{equation}\left.
\parbox{5.3cm}{if $\Psi$ collapses an edge parallel to $[a_0,q_0]$ or $[c_1,q_1]$, then $|H/\Psi|\leq 9$.}\,\,\,\right\}
\label{eqpbxdzGmttCbKsVrl}
\end{equation}

\begin{case}\label{casethree} We assume that no edge collapsed by $\Psi$ is parallel to $[a_0,q_0]$ or $[c_1,q_1]$. If  $[p,r]$ in on the (geometric) line through $a=a_0$ and $a_2$ or through $c=c_1$ and $c_3$, then $[p,r]$ is perspective to a vertical edge on the (vertical) line through $0$ and $q_0$, and this vertical edge is also collapsed by $\Psi$. So we can assume that $[p,r]$ is a vertical edge. We can also assume that $r$ is maximal (with respect to the lattice ordering). If we had that $r=b$, then 
$\Psi$ would collapse the edge $[0,b]=[p,r]$, whereby
$H_m\cong H/\Psi$ would be a homomorphic image of the five-element lattice $\freelat{x,y,z: y\leq x,\,\, y\leq z}$, contradicting $|H_m|\geq 10$. Hence, $[p,r]=[q_{k+1},q_k]$ for some $k\in\Nnul:=\Nplus\cup\set 0$. 
By a \emph{covering pentagon} of $H$ we mean a five-element nonmodular sublattice $\set{o,u,v,w,i}$ such that $o\prec u\prec  w\prec i$ and $o\prec v\prec i$; the covering relation is understood in $H$. A congruence is \emph{nonzero} if it is distinct from the equality relation. It is well known that every nonzero congruence of the pentagon collapses its \emph{monolith edge} $[u,w]$. Hence, whenever the restriction of $\Psi$ to a covering pentagon $\set{o,u,v,w,i}$ is nonzero, then $\Psi$ collapses the monolith edge $[u,w]$ of this pentagon. 
Clearly, if $j\in\Nplus$ is odd, then $\set{c_{j+2},q_{j+2},c_j,q_{j+1},q_j}$ is a covering pentagon with monolith edge $[q_{j+2},q_{j+1}]$. Similarly, if $j\in\Nnul$ is even, then $\set{a_{j+2},q_{j+2},a_j,q_{j+1},q_j}$ is a covering pentagon with monolith edge $[q_{j+2},q_{j+1}]$. Therefore, 
\begin{equation}
\text{for every }j\in\Nnul,  \text{ if }(q_{j},q_{j+1})\in \Psi, \text{ then }
(q_{j+1},q_{j+2})\in \Psi.
\label{eqszhBrlsSzmRTsz}
\end{equation}
It follows from \eqref{eqszhBrlsSzmRTsz} and by the maximality of $r=q_k$ that $\set{q_k, q_{k+1},q_{k+2}, \dots}$ is a block of $\Psi$. So are $\set{a_j: j\geq k\text{ and }j\text{ is even}}$ and  $\set{c_j: j\geq k\text{ and }j\text{ is odd}}$ because of perspectivities; see on the top right of Figure~\ref{figrc}, where $k=7$. Using that Cases~\ref{caseone} and \ref{casetwo} are now excluded as well as $(0,b)\in\Psi$ is, we obtain that the rest of the $\Psi$-blocks are singletons. Hence, it follows that 
$\set{0/\psi, b/\psi, a_{k+1}/\Psi, c_k/\Psi, q_k/\psi}$ is a covering $M_3$ sublattice of $H/\Psi$ if $k$ is odd; see the bottom right of Figure~\ref{figrc}. Similarly, 
$\set{0/\psi, b/\psi, a_{k}/\Psi, c_{k+1}/\Psi, q_k/\psi}$ is a covering $M_3$ sublattice if $k$ is even. Hence, regardless the parity of $k$,  $H/\Psi$ has three atoms such that any two of these atoms have the same join. Since $H_m$ fails to have this property, it cannot be isomorphic to $H/\Psi$. This contradicts 
\eqref{eqtxtpCktvGkltsn} and so Case~\ref{casethree} is excluded.
\end{case}
All the three cases have been excluded. Therefore, we are in the position to conclude  the validity of \eqref{eqtxthbnFnpRsd}.

While dealing with Case~\ref{casethree}, we saw that if a nonzero congruence $\Psi$ is in the scope of this case, then 
$H/\Psi$ is a homomorphic image of the five-element lattice $\freelat{x,y,z: y\leq x,\,\, y\leq z}$ or (up to isomorphism) it belongs to a family of finite lattices; one member of this family is given at the bottom right of Figure~\ref{figrc}. This fact together with \eqref{eqpbxdzGmttCbKsVrl} yield that for every nonzero congruence $\Psi$ of $H$,  the quotient lattice $H/\Psi$ is finite. In particular, then $H/\Psi$ has only finitely many congruences. Using the well-known Correspondence Theorem, see Theorem 6.20 in Burris and Sankappanavar~\cite{bursankap}, we obtain that for every nonzero congruence $\Psi$ of $H$, there are only finitely many congruences larger than $\Psi$. We are in the position to claim that 
\begin{equation}
\text{every congruence of $H$ is finitely generated.}
\label{eqtxtmDnsnGrGLftv}
\end{equation}
Indeed, let $\Psi$ be a congruence of $H$. Since the zero congruence is finitely generated, we can assume that $\Psi$ is nonzero. Pick a pair $(d_1,e_1)\in \Psi$ such that $d_1\neq e_1$. Then $\Psi_1$ is a finitely generated nonzero congruence and $\Psi_1\leq \Psi$. If $\Psi_1:=\con{d_1,e_1}<\Psi$, then pick a pair $(e_2,d_2)\in \Psi\setminus \Psi_1$. If $\Psi_2:=\Psi_1\vee \con{d_2,e_2}<\Psi$, then pick a pair $(e_3,d_3)\in \Psi\setminus \Psi_2$, and so on.
Since there are only finitely many congruences larger than the nonzero congruence $\Psi_1$, we cannot find infinitely many pairs $(d_i,e_i)\in \Psi\setminus\Psi_{i-1}$ in this way.  Hence, $\Psi=\Psi_{i-1}=\con{d_1,e_1}\vee\dots\vee \con{d_{i-1},e_{i-1}}$ for some $i$, proving \eqref{eqtxtmDnsnGrGLftv}.

Finally, lattices are congruence modular since they are even congruence distributive. This fact, \eqref{eqtxtmDnsnGrGLftv}, and the fact that $L$ is a subdirect product of $H$ with itself yield that \eqref{eqpbxRuskucMair} is applicable with $(L,H,H)$ playing the role of $(C,A,B)$.  For the sake of contradiction, suppose that $L$ is finitely presented. Then so is $H$ by
\eqref{eqpbxRuskucMair}. This contradicts \eqref{eqtxthbnFnpRsd} and proves  part \eqref{thmmainc} of the theorem. The proof of Theorem~\ref{thmmain} is complete.
\end{proof}

For possible later reference, we combine \eqref{eqtxtHthRgltNs},   \eqref{eqtxthTsdsrbv}, \eqref{eqtxthbnFnpRsd}, and \eqref{eqtxdhRrngfGcnRcs} as follows.

\begin{remark}\label{remhrrBnNl} The herringbone lattice $H$, see  Figure~\ref{figrc}, is $3$-generated, $2$-distributive, it is not finitely presented, and each of its congruence relations is finitely generated.
\end{remark}

%
%
%

For \emph{free} lattices, the following lemma has often been used; see, for example, Freese,  Je\v zek, and Nation~\cite{fjn} or Gr\"atzer~\cite{gratzer}. Here, we formulate it only for three-element generating sets. Having no reference for not necessarily free lattices, we present its easy proof.

\begin{lemma}\label{lemmadzghrwtNs} 
Let $L$ be a lattice generated by a three-element subset $\set{a_1,a_2,a_3}$. If $i$, $j$, and $k$ are subscripts such that $\set{i,j,k}=\set{1,2,3}$, then the following two assertions and their duals hold.
\begin{enumeratei}
\item\label{lemmadzghrwtNsa} If $a_i\wedge a_j\not\leq a_k$ or, equivalently, $a_i\wedge a_j\neq 0$, then $L$ is the disjoint union of $\filter(a_i\wedge a_j)$ and $\ideal a_k$.
\item\label{lemmadzghrwtNsb} If  $a_i\wedge a_j$ is distinct from $0$, then it is an atom of $L$. 
\end{enumeratei}
\end{lemma} 

\begin{proof}
Since $0=a_i\wedge a_j\wedge a_k$, the condition  $a_i\wedge a_j\not\leq a_k$ is clearly equivalent to $a_i\wedge a_j\neq 0$. Assuming this condition, the filter $\filter(a_i\wedge a_j)$ and the ideal $\ideal a_k$ are obviously disjoint. It is also clear that their union is a sublattice. Since this sublattice contains $a_i$, $a_j$ and $a_k$, it equals $L$, proving \eqref{lemmadzghrwtNsa}. 
Next, for the sake of contradiction, we suppose that $a_i\wedge a_j\neq 0$ but  there is an element $d\in L$ such that $0<d<a_i\wedge a_j$. Since $d\notin \filter (a_i\wedge a_j)$, part \eqref{lemmadzghrwtNsa} implies that $d\in\ideal a_k$. However, then $0<d\leq a_i\wedge a_j\wedge a_k=0$, which is a contradiction.
\end{proof}

\section{The proofs of our observations}\label{sectobsproof}


\begin{proof}[Proof of Observation~\ref{observtr}]
Assume that $L$ has no coatom. Then, by (the dual of) Lemma~\ref{lemmadzghrwtNs}\eqref{lemmadzghrwtNsb}, $a_1\vee a_2=a_1\vee a_3=a_2\vee a_3=1$. There are two cases to consider. First, if 
\begin{equation}a_1\wedge a_2=a_1\wedge a_3=a_2\wedge a_3=0, \label{eqzghmnpwHzbq}
\end{equation}
then $L$ is isomorphic to the five-element non-distributive modular lattice $M_3$, and so $L$ has three atoms.
Second, if \eqref{eqzghmnpwHzbq} fails, then $L$ has an atom by Lemma~\ref{lemmadzghrwtNs}\eqref{lemmadzghrwtNsb}. 
This proves part \eqref{observtra} of Observation~\ref{observtr}.

Next, we deal with part \eqref{observtrb}; note 
that $a_0$, $a_1$, and $a_2$ are pairwise distinct. Assume that $k$ defined in part \eqref{observtrb} is at least $2$. 
Then $a_1\wedge a_2$, $a_1\wedge a_3$, and $a_2\wedge a_3$ are pairwise distinct since otherwise if, say, we had that $a_1\wedge a_2=a_1\wedge a_3$, then $a_1\wedge a_2=a_1\wedge a_3= (a_1\wedge a_2)\wedge (a_1\wedge a_3)= a_1\wedge a_2\wedge a_3=0$ would contradict $k\geq 2$. Hence, we have at least $k$ atoms by Lemma~\ref{lemmadzghrwtNs}\eqref{lemmadzghrwtNsb}, so it suffices to show that every atom is the form of $a_i\wedge a_j$ with $i\neq j$. Without loss of generality, we can assume that $k\geq 2$ is witnessed by $a_1\wedge a_2\neq 0\neq a_1\wedge a_3$. Clearly, $a_1\wedge a_2\not\leq a_3$ and $a_1\wedge a_3\not\leq a_2$.
Let $b\in L$ be an atom such that $b\neq a_1\wedge a_2$ and $b\neq a_1\wedge a_3$.  Then $b\notin \filter(a_1\wedge a_2)$ and $b\notin \filter(a_1\wedge a_3)$, because otherwise  $a_1\wedge a_2<b$ or $a_1\wedge a_3<b$,  contradicting the assumption that $b$ is an atom. Hence, Lemma~\ref{lemmadzghrwtNs}\eqref{lemmadzghrwtNsa} implies that $b\in \ideal a_3$ and $b\in \ideal a_2$. Consequently, $b\leq a_2\wedge a_3$. This is a contradiction if $k=2$, because then $a_2\wedge a_3=0$. Hence there are exactly $k$ atoms if $k=2$. If $k=3$, then $a_2\wedge a_3$ is an atom by  Lemma~\ref{lemmadzghrwtNs}\eqref{lemmadzghrwtNsb}, so $b\leq a_2\wedge a_3$ implies that the only atom  distinct from  $a_1\wedge a_2$ and $a_1\wedge a_3$ is the third atom, $a_2\wedge a_3=b$. Thus, we have exactly three atoms if $k=3$. This completes the argument for part \eqref{observtrb}.

Next, assume that $L$ is modular. 
With $k$ defined in part \eqref{observtrb} of Observation~\ref{observtr}, we can assume that $k\leq 1$, because part  \eqref{observtrb} takes care of the opposite case. Thus, without loss of generality, we can assume that $a_1\wedge a_2=0=a_2\wedge a_3$. The modular lattice freely generated by $\set{x,y,z}$ will be denoted by $\fmth$; see, for example, Birkhoff~\cite[page 64]{Birkhoff}, Crawley and Dilworth~\cite[Figure 17-1]{crawleydilworth}, or Gr\"atzer~\cite[page 84]{gratzer}. Note that $\fmth$ is easy to find on the Internet; see for example, McKeown~\cite{mckeown} for an animated version. We can extend $x\mapsto a_1$, $y\mapsto a_2$, $z\mapsto a_3$ to a surjective homomorphism $\fmth\to L$. Let $\Theta:=\con{0,x\wedge y}\vee\con{0,y\wedge z}$; it is a congruence on $\fmth$. By the 
Homomorphism Theorem and the Correspondence Theorem, see  Theorems 6.12 and 6.20 in Burris and Sankappanavar~\cite{bursankap}, $L$ is a quotient lattice $T/\Psi$  of the lattice $T:=\fmth/\Theta$, which is depicted in Figure~\ref{figrb}. 
If none of $(0,s_1)$ and $(0,s_2)$  belongs to the congruence $\Psi$, then $L=T/\Psi$ has exactly two atoms, $\blokk {s_1}\Psi$ and  $\blokk {s_2}\Psi$. 
If $(0,s_1)\in\Psi$, then $\Psi$ collapses the $M_3$ sublattice and (by the Correspondence Theorem again) $L$ is a quotient lattice of $T':=T/\con{0,s_1}$; see   Figure~\ref{figrb}. Since $T'$ has no four-element antichain, neither has $L$, whereby $L$ has at most three atoms. 
We are left with the case where $(0,s_2)\in\Psi$ but $(0,s_1)\notin\Psi$.  Then $L\cong T''/\Gamma$ where   $T'':=T/\con{0,s_2}$ and $\Gamma:=\Psi/\con{0,s_2}$; see   Figure~\ref{figrb} again.
None of the edges $[0,p]$, $[0,q]$ and $[0,r]$ is collapsed by $\Gamma$ since otherwise $\Gamma$ and $\Psi$ would collapse the $M_3$ sublattice of $T''$ and that of $T$, respectively, and so  $(0,s_1)$ would belong to $\Psi$. Hence, $L$ has exactly three atoms, $p/\Gamma$, $q/\Gamma$, and $r/\Gamma$.
Clearly, $L$ has at least one atom since it is a finite lattice; in fact, $|L|\leq |\fmth|=28$. This proves part \eqref{observtrc}.   and Observation~\ref{observtr}.
\end{proof}

\begin{figure}[ht]
\centerline
{\includegraphics[scale=0.9]{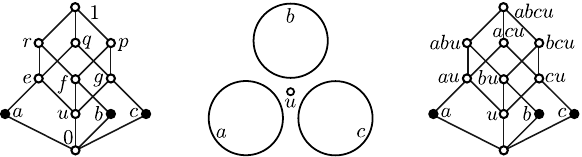}}
\caption{A three-generated lattice with four atoms}
\label{figraa}
\end{figure}

\begin{proof}[Proof of Observation~\ref{observnD}]
\nothing{We present two  easy  ways to see that 
that $L$ given in Figure~\ref{figraa}  
 satisfies the requirements. 
First, referring to the left of Figure~\ref{figraa}, observe that $L\setminus\set b$ is the Hall--Dilworth gluing of the
planar lattice $\set{0,a,u,c,e,g,q}$ and the eight-element boolean lattice $\filter u$. So $L\setminus\set b$ is a lattice. By symmetry, so are $L\setminus\set a$ and  So $L\setminus\set c$. Now if $x,y\in L$, then at least one of 
$L\setminus\set a$, $L\setminus\set b$, and $L\setminus\set c$ contains both $x$ and $y$, and it is easy to check that $x\vee y$ exists in $L$. It follows that $L$ is a lattice. This lattice is clearly generated by the three-element set $\set{a,b,c}$ of black-filled elements, it has four atoms, and it is straightforward to see that it is meet-distributive.} 

\nothing{Second,} 
Let $L$ be the lattice  given in Figure~\ref{figraa}. (In fact, $L$ is diagrammed in the figure twice.)
It is straightforward to verify  
that $L$  satisfies the requirements. 
Alternatively, we can take the four circles in the middle of the diagram. Then, understanding the labels $a$,  \dots, $bcu$, \dots{} as  $\set a$, \dots, $\set{b,c,u}$, \dots,
Cz\'edli~\cite{czgfgcircles} and \cite[Lemma 7.4]{czgcoord} immediately imply that $L$ does the job.
\end{proof}

\end{document}